\documentclass[12pt, draft, reqno]{amsart}
\usepackage{mathpazo}
\usepackage{amsmath, amssymb, amsfonts, amsthm, bm, graphicx, overpic}

\newcommand{\norm}[1]{\lVert#1\rVert}
\newcommand{\ipr}[2]{\left\langle #1, #2 \right\rangle}

\newcommand{\numbersystem}[1]{\mathbb{#1}}

\newcommand{\R}{\numbersystem{R}}

\newcommand{\Sphere}{\mathbb{S}}
\newcommand{\define}[1]{\emph{#1}}

\theoremstyle{plain}
\newtheorem{theorem}{Theorem}
\newtheorem{lemma}{Lemma}
\newtheorem{corollary}[theorem]{Corollary}

\theoremstyle{definition}

\begin{document}

\bibliographystyle{amsplain}

\title{A new proof of V\'azsonyi's conjecture}
\author{Konrad J. Swanepoel}
\thanks{This material is based upon work supported by the South African National Research Foundation.}
\address{Department of Mathematical Sciences, University of South Africa, PO Box 392, Pretoria 0003, South Africa}
\email{swanekj@unisa.ac.za}

\begin{abstract}
We present a self-contained proof that the number of diameter pairs among $n$ points in Euclidean $3$-space is at most $2n-2$.
The proof avoids the ball polytopes used in the original proofs by Gr\"unbaum, Heppes and Straszewicz.
As a corollary we obtain that any three-dimensional diameter graph can be embedded in the projective plane.
\end{abstract}

\maketitle

Let $S$ be a set of $n$ points of diameter $D$ in $\R^d$.
Define the \define{diameter graph} on $S$ by joining all \emph{diameters}, i.e., point pairs at distance $D$.
The following theorem was conjectured by V\'azsonyi, as reported in \cite{Erdos46}.
It was subsequently independently proved by Gr\"unbaum \cite{Grunbaum}, Heppes \cite{Heppes} and Straszewicz \cite{Str}.
\begin{theorem}
The number of edges in a diameter graph on $n\geq 4$ points in $\R^3$ is at most $2n-2$.
\end{theorem}
All three proofs (see \cite[Theorem 13.14]{PachAgarwal}) use the ball polytope obtained by taking the intersection of the balls of radius $D$ centred at the points.
However, these ball polytopes do not behave the same as ordinary polytopes.
In particular, their graphs need not be $3$-connected, as shown by Kupitz, Martini and Perles in \cite{KMP}, where a detailed study of the ball polytopes associated to the above theorem is made.
The proof presented here avoids the use of ball polytopes.
\begin{theorem}\label{maintheorem}
Any diameter graph in $\R^3$ has a bipartite double covering that has a centrally symmetric drawing on the $2$-sphere.
\end{theorem}
In fact, each point $x\in S$ will correspond to an antipodal pair of points $x_r$ and $x_b$ on the sphere, with $x_r$ coloured red and $x_b$ blue.
Each edge $xy$ of the diameter graph will correspond to two antipodal edges $x_ry_b$ and $x_by_r$ on the sphere, giving a properly $2$-coloured graph on $2n$ vertices.
The drawing will be made such that no edges cross.
By Euler's formula there will be at most $4n-4$ edges, hence at most $2n-2$ edges in the diameter graph.
By identifying opposite points of the sphere we further obtain:
\begin{corollary}
Any diameter graph in $\R^3$ can be embedded in the projective plane such that all odd cycles are noncontractible.
\end{corollary}
Therefore, any two odd cycles intersect, and we regain the following theorem of Dol'nikov \cite{Dolnikov}:
\begin{corollary}
Any two odd cycles in a diameter graph on a finite set in $\R^3$ intersect.
\end{corollary}

\begin{proof}[Proof of Theorem~\ref{maintheorem}]
Without loss we assume from now on that $D=1$.
Let $\Sphere^2$ denote the sphere in $\R^3$ with centre the origin and radius $1$.
We may repeatedly remove all vertices of degree at most $1$ in the diameter graph.
Since such vertices can easily be added later, this is no loss of generality.
For each $x\in S$, let $R(x)$ be the intersection of $\Sphere^2$ with the cone generated by $\{y-x : \text{$xy$ is a diameter}\}$.
Each $R(x)$ is a convex spherical polygon with great circular arcs as edges.
(If $x$ has degree $2$ then $R(x)$ is an arc).
Colour $R(x)$ red and $B(x):=-R(x)$ blue.
Assume for the moment the following two properties of these polygons:
\begin{lemma}\label{l1}
If $x\neq y$, then $R(x)$ and $R(y)$ are disjoint.
\end{lemma}
\begin{lemma}\label{l2}
If $R(x)$ and $B(y)$ intersect, then $xy$ is a diameter and $R(x)\cap B(y)=\{y-x\}$.
\end{lemma}
For each $x\in S$ we choose any $x_r$ in the interior of $R(x)$ and let $x_b=-x_r$.
(If $R(x)$ is an arc we let $x_r$ be in its relative interior.)
Draw arcs inside $R(x)$ from $x_r$ to all the vertices of $R(x)$, as well as antipodal arcs from $x_b$ to the vertices of $B(x)$.
This gives a centrally symmetric drawing of a $2$-coloured double covering of the diameter graph.
By Lemmas~\ref{l1} and \ref{l2} no edges cross, and the theorem follows.
\end{proof}
The following proofs of Lemmas~\ref{l1} and \ref{l2} are dimension independent, which gives a double covering on $\Sphere^{d-1}$ of any diameter graph in $\R^d$.
\begin{lemma}\label{l3}
Let $x_1,\dots,x_k$ and $\sum_{i=1}^k\lambda_i x_i$ be unit vectors in $\R^d$, with all $\lambda_i\geq 0$.
Suppose that for some $y\in\R^d$, $\norm{y-x_i}\leq 1$ for all $i=1,\dots,k$.
Then $\norm{y-\sum_{i=1}^k\lambda_i x_i}\leq 1$.
\end{lemma}
\begin{proof}
By the triangle inequality,
\begin{equation}\label{one}
1\leq\norm{\sum_{i=1}^k\lambda_i x_i}\leq\sum_{i=1}^k\lambda_i.
\end{equation}
Expanding $\norm{y-x_i}^2\leq 1$ by inner products,
\begin{equation}\label{two}
-2\ipr{x_i}{y}\leq -\norm{y}^2.
\end{equation}
Therefore,
\begin{align*}
\norm{y-\sum_{i=1}^k\lambda_i x_i}^2 &= \norm{y}^2-2\sum_{i=1}^k \ipr{x_i}{y}+1\\
&\leq \Bigl(1-\sum_{i=1}^k\lambda_i\Bigr)\norm{y}^2+1 \quad\text{by \eqref{two}}\\
&\leq 1 \qquad\text{by \eqref{one}.} \qedhere
\end{align*}
\end{proof}

\begin{proof}[Proof of Lemma~\ref{l1}]
Let the neighbours of $x$ be $x+x_i$, and the neighbours of $y$ be $y+y_j$, with the $x_i$ and $y_j$ unit vectors.
Suppose that
\[\sum_i\lambda_i x_i=\sum_j\mu_j y_j\in R(x)\cap R(y) \text{ with }\lambda_i,\mu_j\geq 0. \]
Since $\norm{x+x_i-y}\leq 1$ for all $i$, Lemma~\ref{l3} gives
\[\norm{x+\sum_i\lambda_i x_i -y}\leq 1.\]
Similarly, Lemma~\ref{l3} applied to $\norm{x-y-y_j}\leq 1$ gives
\[ \norm{x-y-\sum_j\mu_j y_j}\leq 1.\]
By the triangle inequality,
\begin{align*}
2 &=\norm{2\sum_i\lambda_i x_i}\\
&= \norm{(x+\sum_i\lambda_i x_i-y)-(x-y-\sum_j\mu_j y_j)}\\
&\leq \norm{x+\sum_i\lambda_i x_i-y}+\norm{x-y-\sum_j\mu_j y_j}\\
&\leq 2.
\end{align*}
Since we have equality throughout, $x+\sum_i\lambda_i x_i-y$ and $-x+y+\sum_j\mu_j y_j$ are unit vectors in the same direction, hence are equal, which gives $x=y$.
\end{proof}

\begin{proof}[Proof of Lemma~\ref{l2}]
Since $\norm{x_i-x_j}\leq 1$ for all $i, j$, $R(x)$ is properly contained in an open hemisphere of $\Sphere^2$, hence $R(x)\cap B(x)=\emptyset$.
Thus without loss of generality, $x\neq y$.
As before, let the neighbours of $x$ be $x+x_i$, and the neighbours of $y$ be $y+y_j$, with the $x_i$ and $y_j$ unit vectors.
Suppose that $\sum_i\lambda_i x_i=-\sum_j\mu_j y_j\in R(x)\cap B(y)$ with $\lambda_i,\mu_j\geq 0$.
For a fixed $j$ we have that $\norm{x+x_i-y-y_j}\leq 1$ for all $i$.
Lemma~\ref{l3} then gives
\[\norm{x+\sum_i\lambda_i x_i-y-y_j}\leq 1\quad\text{for all $j$}.\]
Again by Lemma~\ref{l3},
\[\norm{x+\sum_i\lambda_i x_i-y-\sum_j\mu_j y_j}\leq 1.\]
By the triangle inequality,
\begin{align*}
2 &=\norm{2\sum_i\lambda_i x_i}\\
&= \norm{(x+\sum_i\lambda_i x_i-y-\sum_j\mu_j y_j)+(y-x)}\\
&\leq \norm{x+\sum_i\lambda_i x_i-y-\sum_j\mu_j y_j}+\norm{y-x}\\
&\leq 2.
\end{align*}
Since we have equality throughout, $x+\sum_i\lambda_i x_i-y-\sum_j\mu_j y_j$ and $y-x$ are unit vectors in the same direction, hence are equal, which gives $x+\sum_i\lambda_i x_i=y$ and $R(x)\cap B(y)=\{y-x\}$.
\end{proof}

\end{document}